\documentclass[11pt]{amsart}
\usepackage[utf8]{inputenc}
\usepackage{setspace}
\usepackage[margin=1in]{geometry}
\usepackage{cite, hyperref, comment, url}
\usepackage{tikz, graphicx, marvosym, subfigure, float}
\usetikzlibrary{decorations.pathreplacing}
\usepackage{dsfont, amsfonts, amssymb, amsmath, amsthm}
\usepackage{algorithm, algpseudocode}
\numberwithin{equation}{section}

\newtheorem{theorem}{Theorem}[section]

\newtheorem{definition}[theorem]{Definition}
\newtheorem{lemma}[theorem]{Lemma}

\renewcommand{\pmod}[1]{\ (\mathrm{mod}\ #1)}

\title{Visibility in hypercubes}

\author{Maria Axenovich}

\address{Maria Axenovich \newline Karlsruhe Institute of Technology, Englerstraße 2, D-76131 Karlsruhe, Germany}
\email{maria.aksenovich@kit.edu}

\author{Dingyuan Liu}

\address{Dingyuan Liu \newline Karlsruhe Institute of Technology, Englerstraße 2, D-76131 Karlsruhe, Germany}
\email{liu@mathe.berlin}

\begin{document}
\maketitle

\vspace{-0.7em}
\begin{abstract}
A subset $M$ of vertices in a graph $G$ is a mutual-visibility set if any two vertices $u$ and $v$ in $M$ ``see'' each other in $G$, that is, there exists a shortest $u,v$-path in $G$ that contains no elements of $M$ as internal vertices. The mutual-visibility number $\mu(G)$ of a graph $G$ is the largest size of a mutual-visibility set in $G$. Let $n\in\mathbb{N}$ and $Q_{n}$ be an $n$-dimensional hypercube. Cicerone, Di Fonso, Di Stefano, Navarra, and Piselli showed that $2^{n}/\sqrt{n}\leq\mu(Q_{n})\leq2^{n-1}$. In this paper, we prove that $\mu(Q_{n})>0.186\cdot2^n$ and thus establish that $\mu(Q_{n})=\Theta(2^{n})$.

We also consider the chromatic mutual-visibility number, $\chi_{\mu}(G)$, defined as the smallest number of colors used on vertices of $G$, such that every color class is a mutual-visibility set in $G$. Klav\v{z}ar, Kuziak, Valenzuela-Tripodoro, and Yero asked whether $\chi_{\mu}(Q_{n})=O(1)$. We answer their question in the negative, namely, we show that $\chi_{\mu}(Q_{n})$ is a growing function of $n$. Moreover, we show that $\chi_{\mu}(Q_{n})=O(\log\log{n})$.

Finally, we study the so-called total mutual-visibility number of graphs and give asymptotically tight bounds on this parameter for hypercubes.
\end{abstract}

\section{Introduction}
Let $G$ be a simple graph and $M\subseteq{V(G)}$. For any two vertices $u,v\in{V(G)}$, we say that $u$ and $v$ are \textit{$M$-visible} if there exists a shortest $u,v$-path in $G$ that contains no vertices from $M\backslash\{u,v\}$. Note that any two adjacent vertices are $M$-visible.
\begin{definition}
\label{def1}
Let $M$ be a subset of vertices in a graph $G$. We say that $M$ is a mutual-visibility set if any two vertices $u,v\in{M}$ are $M$-visible. The largest size of a mutual-visibility set in $G$ is denoted by $\mu(G)$ and called the mutual-visibility number of $G$.
\end{definition}

The systematic investigation of mutual-visibility in graphs was pioneered by Di Stefano\cite{di2022mutual}, and has garnered extensive attention and subsequent research\cite[etc.]{boruzanli2024mutual,brevsar2024lower,cicerone2024cubelike,cicerone2023variety,cicerone2023mutual,cicerone2024mutual,cicerone2024diameter,korvze2024mutual,korvze2024variety} in recent years. The exact value of $\mu(G)$ is only known for very few graph classes, e.g., cliques, complete bipartite graphs, trees, cycles, grids\cite{di2022mutual}, butterflies\cite{cicerone2024cubelike}, Cartesian products of paths and cycles\cite{korvze2024mutual}, Cartesian products of cliques, and direct products of cliques\cite{cicerone2024mutual}. It has been shown in the initial work by Di Stefano\cite{di2022mutual} that determining $\mu(G)$ for any given graph $G$ is NP-complete. By considering the mutual-visibility set consisting of all neighbors of the maximum degree vertex, we can find a large mutual-visibility set in any dense graph. As the mutual-visibility problem was initially motivated by establishing efficient and confidential communication in networks, the research on $\mu(G)$ mainly focuses on sparse and highly connected graphs, such as product graphs and hypercube-like graphs.

The primary objective of this paper is to explore the mutual-visibility phenomenon in hypercubes. For $n\in\mathbb{N}$, the \textit{$n$-dimensional hypercube} $Q_{n}$ is a graph on the vertex set $2^{[n]}$, where two vertices $A,B\in2^{[n]}$ form an edge in $Q_{n}$ if and only if their \textit{symmetric difference} $A\Delta{B}:=(A\backslash{B})\cup(B\backslash{A})$ has size $1$. Cicerone, Di Fonso, Di Stefano, Navarra, and Piselli\cite{cicerone2024cubelike} initially showed that $2^{n}/\sqrt{n}\leq\mu(Q_{n})\leq2^{n-1}$. The upper bound was later improved by Kor\v{z}e and Vesel\cite{korvze2024variety} to $\mu(Q_{n})\leq\frac{59}{128}\cdot2^{n}$ for $n\geq7$. Bodn\'ar\cite{bodnar2025github} recently obtained $\mu(Q_{n})\leq\frac{1916879}{4718592}\cdot2^n$ for sufficiently large $n$ using flag algebras. Here, we prove that $\mu(Q_{n})$ is at least a constant fraction of $2^{n}$ and hence determine $\mu(Q_{n})$ up to a multiplicative constant.
\begin{theorem}
\label{thm1}
For every $n\in\mathbb{N}$, we have $\mu(Q_{n})>\frac{14}{75}\cdot2^{n}$.
\end{theorem}

We derive Theorem \ref{thm1} by establishing a connection between mutual-visibility sets in $Q_n$ and a special class of hypergraphs called daisies. Let $r,s,t\in\mathbb{N}$ with $\min\{r,s\}\geq{t}$. Following the notation of Bollob\'as, Leader, and Malvenuto\cite{bollobas2011daisies}, an \textit{$(r,s,t)$-daisy} $\mathcal{D}_{r}(s,t)$ is a collection of all $r$-element sets $T$ satisfying $A\subseteq{T}\subseteq{B}$, for some given $A\subseteq{B}$ with $\lvert{A}\rvert=r-t$ and $\lvert{B}\rvert=r+s-t$. The \textit{Tur\'{a}n number} $\mathrm{ex}(n,\mathcal{D}_{r}(s,t))$ is the largest size of a set family $\mathcal{F}_r\subseteq\binom{[n]}{r}$ containing no $(r,s,t)$-daisy. By constructing dense mutual-visibility sets in $Q_n$ via daisy-free families (see Section \ref{construction}), we obtain the following lower bound on $\mu(Q_n)$.
\begin{theorem}
\label{thm2}
Let $n, d\in\mathbb{N}$ with $n\geq d\geq 3$. Then
\begin{equation*}
\mu(Q_{n})\geq\frac{1}{d}\left(\sum_{r=d}^{n-d}\mathrm{ex}\left(n,\mathcal{D}_{r}(2d,d)\right) + 2\sum_{r=0}^{d-1} \binom{n}{r}\right).
\end{equation*}
\end{theorem}

Klav\v{z}ar, Kuziak, Valenzuela-Tripodoro, and Yero\cite{klavzar2024color} recently introduced a coloring version of the mutual-visibility problem.
\begin{definition}
\label{chromatic}
Let $G$ be a graph. The smallest number of colors  in a vertex-coloring  of $G$ such that  each color class  is a mutual-visibility set is called the chromatic mutual-visibility number of $G$ and denoted by $\chi_{\mu}(G)$. 
\end{definition}

Equivalently, $\chi_{\mu}(G)$ is the smallest integer such that $V(G)$ can be partitioned into $\chi_{\mu}(G)$ mutual-visibility sets. It is clear that $\chi_{\mu}(G)\geq\lvert{V(G)}\rvert/\mu(G)$. Naturally, one might ask whether $\chi_{\mu}(G)=O\left(\lvert{V(G)}\rvert/\mu(G)\right)$ holds in general. Knowing that $\mu(Q_{n})=\Omega(2^{n})$ from an earlier arXiv version of this paper, Klav\v{z}ar et al.\cite{klavzar2024color} raised the following question:
\begin{center}
\textit{Is there an absolute constant $C>0$, such that $\chi_{\mu}(Q_{n})\leq{C}$ holds for all $n\in\mathbb{N}$?}
\end{center}
We answer their question in the negative with the following result. In particular, this shows that $\chi_{\mu}(G)$ can be arbitrarily far from the trivial lower bound $\lvert{V(G)}\rvert/\mu(G)$.
\begin{theorem}
\label{new}
For any $q>0$ there exists some $n_{0}>0$, such that $\chi_{\mu}(Q_{n})>q$ for all $n\geq n_{0}$. On the other hand, $\chi_{\mu}(Q_{n})=O(\log\log{n})$.
\end{theorem}

We also study a natural variant of mutual-visibility introduced by Cicerone, Di Stefano, Klav\v{z}ar, and Yero\cite{cicerone2024mutual}.
\begin{definition}
\label{def4}
Let $G$ be a graph and $M\subseteq{V(G)}$. We say that $M$ is a total mutual-visibility set if any two vertices $u,v\in{V(G)}$ are $M$-visible. The largest size of a total mutual-visibility set in $G$ is denoted by $\mu_{t}(G)$ and called the total mutual-visibility number of $G$.
\end{definition}

Since a total mutual-visibility set is always a mutual-visibility set, we have $\mu_{t}(G)\leq\mu(G)$. The exact value of $\mu_{t}(G)$ is also known for rather few graph classes, see\cite{boruzanli2024mutual,brevsar2024lower,bujta2023total,cicerone2024mutual,cicerone2023variety,cicerone2024cubelike,kuziak2023total,tian2022graphs}. Furthermore, determining $\mu_{t}(G)$ for any given graph $G$, as demonstrated in\cite{cicerone2023variety}, is also NP-complete. We obtain some tight bounds on $\mu_{t}(Q_{n})$, improving the previous result $\mu_{t}(Q_{n})\geq2^{n-2}/(n^{2}-n)$ by Bujt\'{a}s, Klav\v{z}ar, and Tian\cite{bujta2023total}.
\begin{theorem}
\label{thm4}
For every $n\in\mathbb{N}$, we have $2^{n-1}/n \leq \mu_{t}(Q_{n}) \leq  2^{n}/n$. If $n=2^{m}-1$ with $m\in\mathbb{N}$, then $2^{n}/(n+1) \leq\mu_{t}(Q_{n})\leq 2^{n}/n$.
\end{theorem}

The remainder of this paper is organized as follows. In Section \ref{construction}, we exhibit our construction of a dense mutual-visibility set in $Q_{n}$ and prove the key lemma. Section \ref{proofofthm2} is dedicated to proving Theorem \ref{thm2}. We derive Theorem \ref{thm1} from Theorem \ref{thm2} in Section \ref{proofofthm1}. In Section \ref{coloring} we establish Theorem \ref{new}. Finally, we prove Theorem \ref{thm4} in Section \ref{totalmutual}. Section \ref{conclusions} contains concluding remarks and open questions.

\section{Mutual-visibility in $Q_{n}$}
\label{mutual}
Given $n\in\mathbb{N}$ and $r\in \{0, 1, \dots, n\}$, the \textit{$r^{\mathrm{th}}$ layer} $\mathcal{L}_{r}$ of the hypercube $Q_n$ is $\binom{[n]}{r}$, i.e., the family of all $r$-subsets of $[n]=\{1, \dots, n\}$. We say that the layer $\mathcal{L}_r$ is \textit{between} the layers $\mathcal{L}_{r'}$ and $\mathcal{L}_{r''}$ if $r'<r<r''$ or $r''<r<r'$. We denote by $\mathrm{dist}(A,B)$ the distance in $Q_n$ between vertices $A$ and $B$. Cicerone et al.\cite[Theorem 1]{cicerone2024cubelike} observed that for any $r$, $\mathcal{L}_{r}\cup \mathcal{L}_{r+d}$ with $d\geq 3$ is a mutual-visibility set in $Q_n$, giving a lower bound $\mu(Q_n) \geq \Omega(2^n/\sqrt{n})$. Here we construct a larger mutual-visibility set.

\subsection{Construction of a mutual-visibility set}
\label{construction}
Let $d\geq 3$. Let $\mathcal{F}_r$ be a $\mathcal{D}_r(2d,d)$-free subfamily of $\mathcal{L}_r$ for each $r\in \{d, \dots, n-d\}$ and let $\mathcal{F}_r=\mathcal{L}_r$ for $r\in\{0,\dots,d-1\}\cup\{n-d+1,\dots,n\}$. Let
\begin{equation*}
M(\lambda) = \bigcup_{r\equiv\lambda\pmod{d}} \mathcal{F}_r.    
\end{equation*}
We say that the layers of $Q_n$ containing vertices of $M(\lambda)$ are \textit{selected layers}. See Figure \ref{main-construction} for an illustration of $M(\lambda)$.
\begin{figure}[h] 
\centering
\begin{tikzpicture}[scale=0.7]
\node at (0,4.3){$[n]$};
\node at (0,-4.3){$\emptyset$};

\draw[ultra thick,darkgray] (-2.3,0)--(2.3,0);
\fill[white] (-2,0) circle (1.5pt);
\fill[white] (-1,0) circle (1.5pt);
\fill[white] (0,0) circle (1.5pt);
\fill[white] (1,0) circle (1.5pt);
\fill[white] (2,0) circle (1.5pt);
\node at (3,0){$\mathcal{F}_{r}$};

\draw[dashed,gray] (-2.2,0.3)--(2.2,0.3);
\draw[dashed,gray] (-2.2,0.6)--(2.2,0.6);
\draw[ultra thick,darkgray] (-2.23,0.9)--(2.23,0.9);
\fill[white] (-1.9,0.9) circle (1.5pt);
\fill[white] (-0.9,0.9) circle (1.5pt);
\fill[white] (-0.1,0.9) circle (1.5pt);
\fill[white] (0.9,0.9) circle (1.5pt);
\fill[white] (1.9,0.9) circle (1.5pt);
\node at (3,0.9){$\mathcal{F}_{r+d}$};

\draw[dashed,gray] (-2.03,1.2)--(2.05,1.2);
\draw[dashed,gray] (-2.03,1.5)--(2.03,1.5);
\draw[ultra thick,darkgray] (-2.03,1.8)--(2.03,1.8);
\fill[white] (-1.7,1.8) circle (1.5pt);
\fill[white] (-0.6,1.8) circle (1.5pt);
\fill[white] (0.3,1.8) circle (1.5pt);
\fill[white] (1.7,1.8) circle (1.5pt);
\node at (3,2.3){$\vdots$};

\draw[dashed,gray] (-1.8,2.1)--(1.83,2.1);
\draw[dashed,gray] (-1.7,2.4)--(1.73,2.4);
\draw[ultra thick,darkgray] (-1.63,2.7)--(1.63,2.7);
\fill[white] (-1.2,2.7) circle (1.5pt);
\fill[white] (-0.4,2.7) circle (1.5pt);
\fill[white] (0.5,2.7) circle (1.5pt);
\fill[white] (1.2,2.7) circle (1.5pt);

\draw[dashed,gray] (-1.45,3)--(1.45,3);
\draw[dashed,gray] (-1.1,3.3)--(1.12,3.3);
\draw[ultra thick,darkgray] (-0.79,3.6)--(0.79,3.6);

\draw[dashed,gray] (-2.2,-0.3)--(2.2,-0.3);
\draw[dashed,gray] (-2.2,-0.6)--(2.2,-0.6);
\draw[ultra thick,darkgray] (-2.23,-0.9)--(2.23,-0.9);
\fill[white] (-1.9,-0.9) circle (1.5pt);
\fill[white] (-0.9,-0.9) circle (1.5pt);
\fill[white] (-0.1,-0.9) circle (1.5pt);
\fill[white] (0.9,-0.9) circle (1.5pt);
\fill[white] (1.9,-0.9) circle (1.5pt);
\node at (3,-0.9){$\mathcal{F}_{r-d}$};

\draw[dashed,gray] (-2.03,-1.2)--(2.05,-1.2);
\draw[dashed,gray] (-2.03,-1.5)--(2.03,-1.5);
\draw[ultra thick,darkgray] (-2.03,-1.8)--(2.03,-1.8);
\fill[white] (-1.7,-1.8) circle (1.5pt);
\fill[white] (-0.6,-1.8) circle (1.5pt);
\fill[white] (0.3,-1.8) circle (1.5pt);
\fill[white] (1.7,-1.8) circle (1.5pt);
\node at (3,-2){$\vdots$};

\draw[dashed,gray] (-1.8,-2.1)--(1.83,-2.1);
\draw[dashed,gray] (-1.7,-2.4)--(1.73,-2.4);
\draw[ultra thick,darkgray] (-1.63,-2.7)--(1.63,-2.7);
\fill[white] (-1.2,-2.7) circle (1.5pt);
\fill[white] (-0.4,-2.7) circle (1.5pt);
\fill[white] (0.5,-2.7) circle (1.5pt);
\fill[white] (1.2,-2.7) circle (1.5pt);

\draw[dashed,gray] (-1.45,-3)--(1.45,-3);
\draw[dashed,gray] (-1.1,-3.3)--(1.12,-3.3);
\draw[ultra thick,darkgray] (-0.79,-3.6)--(0.79,-3.6);

\draw[thick,black] (0,0) ellipse (6em and 10em);
\end{tikzpicture}
\caption{Construction of a dense mutual-visibility set}
\label{main-construction}
\end{figure}

\begin{lemma}
\label{lem:lem2}
The set $M=M(\lambda)$ is a mutual-visibility set in $Q_n$ for any $\lambda\in\mathbb{Z}$.
\end{lemma}
\begin{proof}[Proof of Lemma \ref{lem:lem2}]
For any two vertices $A, B\in M$ we shall verify that they are $M$-visible. Cicerone et al.\cite[Theorem 1]{cicerone2024cubelike} observed the following.

\vspace{1em}
\textit{Claim 1.} For any two vertices $A,B\in M \cap (\mathcal{L}\cup \mathcal{L}')$, where $\mathcal{L}$ and $\mathcal{L}'$ are two consecutive selected layers, there exists a shortest $A,B$-path in $Q_{n}$ that is internally disjoint from $\mathcal{L}\cup \mathcal{L}'$ and only goes through the layers between $\mathcal{L}$ and $\mathcal{L}'$.

\vspace{1em}
A proof of Claim 1 is included in the appendix for completeness. This implies that any two vertices of $M$ that are either from the same layer or from some consecutive selected layers ``see" each other. For any other two vertices of $M$, we shall argue that there exists a shortest path between them going through the ``holes'' in the selected layers in between, see Figure \ref{fig5}.
\begin{figure}[h]
\centering
\begin{tikzpicture}[scale=7/10]
\node at (-7.8,-1.5){\small$\mathcal{F}_{r}$};
\node at (-7.8,-0.5){\small$\mathcal{F}_{r'}$};
\node at (-7.8,0.5){\small$\mathcal{F}_{r_{3}}$};
\node at (-7.8,1.4){\small$\vdots$};
\node at (-7.8,2){\small$\mathcal{F}_{r_{k}}$};

\draw[ultra thick,gray] (-7,-1.5)--(10,-1.5);
\draw[dashed,gray] (-7,-1.2)--(10,-1.2);
\draw[dashed,gray] (-7,-0.9)--(10,-0.9);
\draw[ultra thick, gray] (-7,-0.5)--(10,-0.5);
\draw[dashed,gray] (-7,-0.2)--(10,-0.2);
\draw[dashed,gray] (-7,0.1)--(10,0.1);
\draw[ultra thick,gray] (-7,0.5)--(10,0.5);
\draw[dashed,gray] (-7,1)--(10,1);
\draw[dashed,gray] (-7,1.3)--(10,1.3);
\draw[dashed,gray] (-7,1.6)--(10,1.6);
\draw[ultra thick, gray] (-7,2)--(10,2);

\fill[white] (-3,-0.5) circle (3pt);
\fill[white] (0,0.5) circle (3pt);
\draw[thick,blue] (-6,-1.5)--(-5.5,-1.2)--(-5.2,-0.9)--(-4.9,-1.2)--(-4.6,-0.9)--(-4.3,-1.2)--(-4,-0.9)--(-3.7,-1.2)--(-3,-0.5)--(-2.5,-0.2)--(-2.3,0.1)--(-2,-0.2)--(-1.7,0.1)--(-1.4,-0.2)--(-1.1,0.1)--(-0.8,-0.2)--(0,0.5)--(0.8,1)--(1.6,0.8)--(2.4,1.3)--(3.2,1.1)--(4,1.6)--(4.8,1.4)--(6,2);

\fill (-6,-1.5) circle (3pt);
\node at (-6,-1.8){\small$A$};
\fill (6,2) circle (3pt);
\node at (6,2.3){\small$B$};
\end{tikzpicture}
\caption{A shortest $A,B$-path going through ``holes'' in the layers}
\label{fig5}
\end{figure}

\vspace{1em}
\textit{Claim 2.} For any $A,B\in2^{[n]}$, where $\lvert{A}\rvert \leq r-d$ and $\lvert{B}\rvert\geq r+d$, there exists $C\in\mathcal{L}_{r}\backslash\mathcal{F}_{r}$ such that $A\cap{B}\subseteq{C}\subseteq{A}\cup{B}$.

Indeed, since $\lvert{A\cap{B}}\rvert\leq\lvert{A}\rvert\leq{r-d}$ and $\lvert{A\cup{B}}\rvert\geq\lvert{B}\rvert\geq{r+d}$, there are some $A'\in\mathcal{L}_{r-d}$ and $B'\in\mathcal{L}_{r+d}$ with $A\cap{B}\subseteq{A'}\subseteq{B'}\subseteq{A\cup{B}}$. Since $\mathcal{F}_r$ is $\mathcal{D}_r(2d, d)$-free, for any $A'\in \mathcal{L}_{r-d}$ and $B'\in \mathcal{L}_{r+d}$ with $A'\subseteq B'$, the family $\mathcal{F}_r$ omits some $C\in \mathcal{L}_r$ satisfying $A'\subseteq C \subseteq B'$. In particular, $C\in\mathcal{L}_{r}\backslash\mathcal{F}_{r}$ and $A\cap{B}\subseteq{C}\subseteq{A}\cup{B}$. This proves Claim 2.

\vspace{1em}
\textit{Claim 3.} For any two vertices $A,B$ of $Q_n$ and any $C$ such that $A\cap B \subseteq C \subseteq A\cup B$, there is a shortest $A,B$-path in $Q_n$ containing $C$.

Note that $\mathrm{dist}(A,B)=\lvert{A\Delta B}\rvert$. Let $C$ be an arbitrary set with $A\cap B \subseteq C \subseteq A\cup B$. Then $\mathrm{dist}(A,C) =|A\Delta C|$ and $\mathrm{dist}(C,B)=|C\Delta B|$. Thus there is an $A,B$-walk $W$ through $C$ of length $|A\Delta C|+|C\Delta B|$. We have that $|A\Delta C|+|C\Delta B|= \left(\lvert{A\backslash C}\rvert+\lvert{C\backslash B}\rvert\right)+\left(\lvert{B\backslash C}\rvert+\lvert{C\backslash A}\rvert\right) =
\lvert{A\backslash B}\rvert + \lvert{B\backslash A}\rvert = \lvert{A\Delta B}\rvert$. Since the length of $W$ is $\mathrm{dist}(A,B)$, $W$ is a shortest $A,B$-path. This completes the proof of Claim 3.

\vspace{1em}
Now, consider $A\in\mathcal{F}_r$ and $B\in\mathcal{F}_{r'}$ for $r'\geq r$. Recall that $r'\equiv r\pmod{d}$. We shall prove by induction on $r'-r$, that there is a shortest $A,B$-path in $Q_n$, whose internal vertices are not in $M$ and are contained in layers between $\mathcal{L}_r$ and $\mathcal{L}_{r'}$.
For $r'-r=0$ or $r'-r=d$, such a path exists by Claim 1. If $r'-r\geq2d$, let $C\in\mathcal{L}_{r+d}\backslash\mathcal{F}_{r+d}$ such that $A\cap B\subseteq C\subseteq A\cup B$. Such a $C$ exists by Claim 2.

By induction, there is a shortest (in $Q_n$) $A,C$-path that is internally disjoint from $M$ and goes through the layers between $\mathcal{L}_{r}$ and $\mathcal{L}_{r+d}$. Similarly there is a shortest (in $Q_n$) $C,B$-path that is internally disjoint from $M$ and goes through the layers between $\mathcal{L}_{r+d}$ and $\mathcal{L}_{r'}$. The union $P$ of these two paths is an $A,B$-path  whose internal vertices are not in $M$. Moreover, by Claim 3 there is a shortest $A,B$-path $P'$ in $Q_n$ that contains $C$. We see that the length of the $A,C$-subpath of $P$ is at most the length of the $A,C$-subpath of $P'$. Similarly, the $C,B$-subpath of $P$ has length at most the length of $C,B$-subpath of $P'$. Hence $P$ is not longer than $P'$, i.e., $P$ is a desired shortest $A,B$-path.  
\end{proof}

\subsection{Proof of Theorem \ref{thm2}}
\label{proofofthm2}
\begin{proof}[Proof of Theorem \ref{thm2}]
By Lemma \ref{lem:lem2} we see that $Q_n$ contains pairwise disjoint mutual-visibility sets $M(0)$, $M(1)$, \ldots, $M(d-1)$ whose total size is $\sum_{r=d}^{n-d} \mathrm{ex}(n,\mathcal{D}_{r}(2d,d)) + 2\sum_{r=0}^{d-1} \binom{n}{r}$. Thus some $M(i)$ has size at least $\frac{1}{d}\left( \sum_{r=d}^{n-d} \mathrm{ex}(n,\mathcal{D}_{r}(2d,d)+ 2\sum_{r=0}^{d-1} \binom{n}{r}\right)$. 
\end{proof}

\subsection{Proof of Theorem \ref{thm1}}
\label{proofofthm1}
Let $r,s,t\in\mathbb{N}$ with $\min\{r,s\}\geq{t}$. The \textit{Tur\'{a}n density} of $(r,s,t)$-daisies is defined as
\begin{equation*}
\pi\left(\mathcal{D}_{r}(s,t)\right):=\lim_{n\to\infty}\frac{\mathrm{ex}\left(n,\mathcal{D}_{r}(s,t)\right)}{\binom{n}{r}}.
\end{equation*}

We shall use the following statement from the proof of Theorem 5 in the paper by Ellis, Ivan, and Leader\cite{ellis2024daisies}. Let $q\in\mathbb{N}$ be a fixed prime power. Then
\begin{equation}
\label{daisy}
\lim_{r\to\infty} \pi(n, D_{r}(q+2,2))\geq\prod_{k=1}^{\infty}(1-q^{-k}).
\end{equation}

We also need the following monotonicity statement mentioned in the literature, see, e.g., Keevash\cite{keevash2011survey} and Sidorenko\cite{sidorenko2024turan} and we include its proof in the appendix.
\begin{lemma}
\label{lem3}
Let $r,s,t$ be positive integers with $\min\{r,s\}\geq{t}$. Then for all $n\geq{r}$
\begin{equation*}
\frac{\mathrm{ex}\left(n,\mathcal{D}_{r}(s,t)\right)}{\binom{n}{r}}\geq\frac{\mathrm{ex}\left(n+1,\mathcal{D}_{r}(s,t)\right)}{\binom{n+1}{r}} \quad \text{and} \quad \pi\left(D_{r}(s,t)\right)\geq\pi\left(\mathcal{D}_{r+1}(s,t)\right).
\end{equation*}
\end{lemma}

\begin{proof}[Proof of Theorem \ref{thm1}]
Since the stated lower bound is trivial when $n\leq 3$, without loss of generality we assume $n>3$. 
We shall derive Theorem \ref{thm1} from Theorem \ref{thm2} with $d=3$. For that we need to find the lower bounds on $\mathrm{ex}(n,\mathcal{D}_{r}(6,3))$. Since every $\mathcal{D}_{r}(6,3)$ contains a copy of $\mathcal{D}_{r}(5,2)$, we have $\mathrm{ex}(n,\mathcal{D}_{r}(6,3))\geq \mathrm{ex}(n,\mathcal{D}_{r}(5,2))$. Applying \eqref{daisy} with $q=3$ gives
\begin{equation*}
\lim_{r\to\infty}\pi\left(\mathcal{D}_{r}(5,2)\right)\geq\prod_{k=1}^{\infty}(1-3^{-k})>0.56.
\end{equation*}
Accordingly, for $n\geq r\geq2$ Lemma \ref{lem3} implies that $\pi\left(\mathcal{D}_{r}(5,2)\right)>0.56$ and $\mathrm{ex}\left(n,\mathcal{D}_{r}(5,2)\right)>0.56\binom{n}{r}$. Now by Theorem \ref{thm2} with $d=3$, we obtain
\begin{equation*}
\mu(Q_{n})\geq  \frac{1}{3}\left(\sum_{r=3}^{n-3} \mathrm{ex}(n,\mathcal{D}_{r}(6,3))+ 2\sum_{r=0}^{2} \binom{n}{r}\right)\geq   \frac{1}{3}\left(\sum_{r=3}^{n-3} \mathrm{ex}(n,\mathcal{D}_{r}(5,2))+ 2\sum_{r=0}^{2} \binom{n}{r}\right)>  \frac{0.56}{3}  \cdot2^{n}.\qedhere  
\end{equation*}
\end{proof}

\section{Mutual-visibility coloring of $Q_{n}$}
\label{coloring}
\begin{proof}[Proof of Theorem \ref{new}]
Fix any positive integer $q$. We first show that there exists some $n_{0}>0$, such that $\chi_{\mu}(Q_{n})>q$ holds for all $n\geq{n_{0}}$. Given integers $s\geq{r}\geq1$, the \textit{$q$-color hypergraph Ramsey number} $R_{r}(s;q)$ is defined as the smallest integer $N$ such that any $q$-coloring of $\binom{[N]}{r}$ contains a monochromatic copy of $\binom{[s]}{r}$. Since $q$ is fixed in the beginning, we simply write $R_{r}(s)=R_{r}(s;q)$. Now let $n_0=q\cdot\left(R_{2}\circ{R_{3}}\circ\dots\circ{R_{2q}(2q)}\right)$. Let $n\geq{n_{0}}$ and  fix an arbitrary $q$-coloring of $V(Q_{n})$. Following an approach from \cite{axenovich2017boolean}, we iteratively apply the hypergraph Ramsey theorem and see that there is a sequence of sets $X_{2q}\subseteq{X_{2q-1}}\subseteq\dots\subseteq{X_{2}}\subseteq{X_{1}}\subseteq[n]$, such that $\binom{X_{r}}{r}$ is monochromatic for each $r\in[2q]$ and $\lvert{X_{2q}}\rvert=2q$. Let $Q$ be the subgraph of $Q_{n}$ induced by $2^{X_{2q}}$, in particular, $Q$ is a copy of $Q_{2q}$. We have that every layer of $Q$ is monochromatic. Since $Q$ contains $2q+1$ layers and there are only $q$ colors, by the pigeonhole principle at least three layers of $Q$ receive the same color. Let $\mathcal{L}_{i}\cap V(Q)$, $\mathcal{L}_{j}\cap V(Q)$, and $\mathcal{L}_{k}\cap V(Q)$ be the three layers of $Q$ contained in the same color class that we denote by $M$, where $i<j<k$. Consider two vertices $A\in\mathcal{L}_{i}\cap V(Q)$ and $B\in\mathcal{L}_{k}\cap V(Q)$ with $A\subseteq{B}$. Observe that every shortest $A,B$-path in $Q_n$ goes through some vertex $C\in \mathcal{L}_j$ satisfying $A\subseteq{C}\subseteq{B}$. As $Q$ is a copy of a hypercube, all such vertices $C$ are contained in $\mathcal{L}_{j}\cap V(Q)$. Namely, every shortest $A$-$B$ path must go through the layer $\mathcal{L}_{j}\cap V(Q)$, so $M$ is not a mutual-visibility set. This implies that $\chi_{\mu}(Q_{n})>q$.

\vspace{1em}
Now we shall prove the upper bound on $\chi_\mu(Q_n)$. All logarithms below are base $2$. Assume that $n$ is sufficiently large and let $d=\lfloor{\log\log{n}}\rfloor\geq3$. Consider a layer $\mathcal{L}_r$ and let $\lambda \in \{0, 1, \ldots, d-1\}$ such that $r \equiv \lambda \pmod{d}$. If $d\leq r\leq n-d$, we color $\mathcal{L}_r$ uniformly at random in two colors $(\lambda, 1)$ and $(\lambda,2)$. If $r\leq d-1$ or $r\geq n-d+1$ we color $\mathcal{L}_r$ simply with one color $(\lambda,1)$. See Figure \ref{coloring} for an illustration.
\begin{figure}[h]
\centering
\begin{tikzpicture}[scale=0.7]
\node at (0,4.3){$[n]$};
\node at (0,-4.3){$\emptyset$};

\draw[ultra thick,blue] (-2.3,0)--(-2,0);
\draw[ultra thick,red] (-2,0)--(-1,0);
\draw[ultra thick,blue] (-1,0)--(0,0);
\draw[ultra thick,red] (0,0)--(1,0);
\draw[ultra thick,blue] (1,0)--(2,0);
\draw[ultra thick,red] (2,0)--(2.3,0);
\node at (3,0){$\mathcal{F}_{r}$};

\draw[ultra thick,orange] (-2.3,0.3)--(-1.5,0.3);
\draw[ultra thick,violet] (-1.5,0.3)--(-0.8,0.3);
\draw[ultra thick,orange] (-0.8,0.3)--(0,0.3);
\draw[ultra thick,violet] (0,0.3)--(0.8,0.3);
\draw[ultra thick,orange] (0.8,0.3)--(1.5,0.3);
\draw[ultra thick,violet] (1.5,0.3)--(2.3,0.3);

\draw[ultra thick,olive] (-2.3,0.6)--(-1.9,0.6);
\draw[ultra thick,teal] (-1.9,0.6)--(-0.9,0.6);
\draw[ultra thick,olive] (-0.9,0.6)--(-0.1,0.6);
\draw[ultra thick,teal] (-0.1,0.6)--(0.9,0.6);
\draw[ultra thick,olive] (0.9,0.6)--(1.9,0.6);
\draw[ultra thick,teal] (1.9,0.6)--(2.3,0.6);

\node at (0,1.15){$\vdots$};
\draw[ultra thick,blue] (-2.15,1.4)--(-1.7,1.4);
\draw[ultra thick,red] (-1.7,1.4)--(-0.6,1.4);
\draw[ultra thick,blue] (-0.6,1.4)--(0.3,1.4);
\draw[ultra thick,red] (0.3,1.4)--(1.7,1.4);
\draw[ultra thick,blue] (1.7,1.4)--(2.15,1.4);
\node at (3,1.4){$\mathcal{F}_{r+d}$};

\draw[ultra thick,orange] (-2.07,1.7)--(-1.5,1.7);
\draw[ultra thick,violet] (-1.5,1.7)--(-0.8,1.7);
\draw[ultra thick,orange] (-0.8,1.7)--(0,1.7);
\draw[ultra thick,violet] (0,1.7)--(0.8,1.7);
\draw[ultra thick,orange] (0.8,1.7)--(1.5,1.7);
\draw[ultra thick,violet] (1.5,1.7)--(2.07,1.7);

\draw[ultra thick,olive] (-1.97,2)--(-1.4,2);
\draw[ultra thick,teal] (-1.4,2)--(-0.9,2);
\draw[ultra thick,olive] (-0.9,2)--(-0.1,2);
\draw[ultra thick,teal] (-0.1,2)--(0.9,2);
\draw[ultra thick,olive] (0.9,2)--(1.4,2);
\draw[ultra thick,teal] (1.4,2)--(1.97,2);
\node at (0,3){$\vdots$};

\draw[ultra thick,blue] (-2.15,-1.4)--(-1.7,-1.4);
\draw[ultra thick,red] (-1.7,-1.4)--(-0.6,-1.4);
\draw[ultra thick,blue] (-0.6,-1.4)--(0.3,-1.4);
\draw[ultra thick,red] (0.3,-1.4)--(1.7,-1.4);
\draw[ultra thick,blue] (1.7,-1.4)--(2.15,-1.4);
\node at (3,-1.4){$\mathcal{F}_{r-d}$};

\draw[ultra thick,orange] (-2.2,-1.1)--(-1.5,-1.1);
\draw[ultra thick,violet] (-1.5,-1.1)--(-0.8,-1.1);
\draw[ultra thick,orange] (-0.8,-1.1)--(0,-1.1);
\draw[ultra thick,violet] (0,-1.1)--(0.8,-1.1);
\draw[ultra thick,orange] (0.8,-1.1)--(1.5,-1.1);
\draw[ultra thick,violet] (1.5,-1.1)--(2.2,-1.1);

\draw[ultra thick,olive] (-2.25,-0.8)--(-1.9,-0.8);
\draw[ultra thick,teal] (-1.9,-0.8)--(-0.9,-0.8);
\draw[ultra thick,olive] (-0.9,-0.8)--(-0.1,-0.8);
\draw[ultra thick,teal] (-0.1,-0.8)--(0.9,-0.8);
\draw[ultra thick,olive] (0.9,-0.8)--(1.9,-0.8);
\draw[ultra thick,teal] (1.9,-0.8)--(2.25,-0.8);
\node at (0,-0.25){$\vdots$};
\node at (0,-2.3){$\vdots$};

\draw[thick,black] (0,0) ellipse (6em and 10em);
\end{tikzpicture}
\caption{Each color class is a mutual-visibility set}
\label{coloring}
\end{figure}

We claim that with positive probability every color class is a mutual-visibility set. By Lemma \ref{lem:lem2}, it is sufficient to prove that for each $r$ with $d\leq r\leq n-d$, with positive probability every color class in $\mathcal{L}_r$ is $D_r(2d,d)$-free. Fix such an $r$ and let $\mathcal{J}=\mathcal{J}_r(d)$ be the collection of all copies of daisies $D_r(2d,d)$ in the layer $\mathcal{L}_r$.

For any $J\in \mathcal{J}$, let $\mathcal{E}(J)$ be the event that $J$ is monochromatic. Let $p=p(J)$ be the probability that $\mathcal{E}(J)$ occurs. Since $\lvert{J}\rvert=\binom{2d}{d}$,
\begin{equation*}
p=2^{-\lvert{J}\rvert+1}=2^{-\binom{2d}{d}+1}.
\end{equation*}
Let $g=g(J)=\lvert\left\{J'\in \mathcal{J}\backslash \{J\}:\,J'\cap J\neq \emptyset\right\}\rvert$. 
We have that 
\begin{equation*}
g \leq \binom{2d}{d}\binom{r}{d}\binom{n-r}{d}\leq\binom{2d}{d}\binom{n}{2d}<\left(\frac{en}{d}\right)^{2d}.    
\end{equation*}
Indeed, each member of $J$ belongs to $\binom{r}{d}\binom{n-r}{d}$ distinct copies of $D_r(2d,d)$. Thus $\mathcal{E}(J)$ is mutually independent of all but at most $g$  other events $\mathcal{E}(J')$, $J'\in \mathcal{J}$.
Now since
\begin{equation*}
e\cdot p\cdot(g+1)\leq{2^{1-\binom{2d}{d}}}n^{2d}< 2^{-(\log{n})^{2}/\sqrt{100\log\log{n}}+2\log{n}\log\log{n}}<1,    
\end{equation*}
by Lov\'{a}sz Local Lemma\cite{erdos1973local} (see also\cite[Theorem 1.5]{spencer1977local}), with positive probability there exists no monochromatic copy of $D_r(2d,d)$ in $\mathcal{L}_r$.
\end{proof}

\vspace{0.7em}\noindent
\textit{Remark.} Although the proof above shows that $\chi_{\mu}(Q_{n})$ grows with $n$, the dependence of the lower bound on $n$ is quite unsatisfactory. In fact, it is even difficult to express the lower bound on $\chi_{\mu}(Q_{n})$ in terms of $n$, since for $\chi_{\mu}(Q_{n})>q$ we require $n$ to be a composition of $q$-color hypergraph Ramsey numbers.

\section{Total mutual-visibility in $Q_{n}$}
\label{totalmutual}
In this section it is helpful to regard the hypercube $Q_{n}$ as a graph on the vertex set $\{0,1\}^{n}$. For $x,y\in\{0,1\}^{n}$, the \textit{Hamming distance} between $x$ and $y$, corresponding to the distance in the graph $Q_n$, denoted $\mathrm{dist}(x,y)$, is the number of coordinates where $x$ and $y$ differ. For a set $M\subseteq \{0,1\}^{n}$, we say that $M$ \textit{avoids Hamming distance $d$} if for any two vertices $x,y\in M$, $\mathrm{dist}(x,y)\neq d$. Bujt\'{a}s et al.\cite{bujta2023total} showed the following characterization of a total mutual-visibility set in $Q_{n}$.
\begin{lemma}[{\hspace{-0.3mm}\cite[Theorem 6]{bujta2023total}}]
\label{lem5}
$M\subseteq\{0,1\}^{n}$ is a total mutual-visibility set in $Q_{n}$ if and only if $M$ avoids the Hamming distance $2$.
\end{lemma}

By taking a random subset of $\{0,1\}^{n}$ and applying the deletion method to the pairs of vertices with Hamming distance $2$, Bujt\'{a}s et al.\cite{bujta2023total} proved the lower bound $\mu_{t}(Q_{n})\geq2^{n-2}/(n^{2}-n)$ for all $n\geq2$. The same bound (with a slightly better constant) can also be derived from a greedy algorithm. On the other hand, our improvement on the lower bound comes from a coding theory perspective.

A subset $M$ of $\{0,1\}^{n}$ is a \textit{single error correcting code}, if the Hamming distance between any two elements of $M$ is at least $3$. Let for $x\in\{0,1\}^{n}$, $\mathcal{B}(x,1)$ be the set of all elements from $\{0,1\}^{n}$ at Hamming distance at most $1$ from $x$, i.e., a unit ball with center $x$. Note that $\lvert\mathcal{B}(x,1)\rvert=n+1$ and if $M$ is a single error correcting code, then for any distinct $x,y\in{M}$, $\mathcal{B}(x,1)\cap\mathcal{B}(y,1)=\emptyset$. Thus, the size of $M$ is at most $2^{n}/(n+1)$. In particular, if $\lvert{M}\rvert=2^{n}/(n+1)$, then $M$ is called a \textit{perfect single error correcting code}. Hamming\cite{hamming1950error} showed that perfect codes exist for certain values of $n$.
\begin{lemma}[Hamming\cite{hamming1950error}]
\label{lem6}
Let $n=2^{m}-1$ for some integer $m\geq2$. Then there exists a perfect single error correcting code $M\subseteq\{0,1\}^{n}$.
\end{lemma}

The \textit{weight} of an element $x\in\{0,1\}^{n}$ is defined to be the number of $1$'s in $x$. Given $w\in[0,n]$, let $A(n,4,w)$ denote the largest size of a set $M$ of elements from $\{0,1\}^{n}$ with weight $w$ and pairwise Hamming distance at least $4$. Employing a simple combinatorial argument, Graham and Sloane\cite{graham1980lower} proved the following lower bound on $A(n,4,w)$. We include the proof here for completeness and later discussion.
\begin{lemma}[{\hspace{-0.3mm}\cite[Theorem 1]{graham1980lower}}]
\label{lem7}
Let $n\in\mathbb{N}$ and $w\in[0,n]$. Then $A(n,4,w)\geq\binom{n}{w}/n$.
\end{lemma}
\begin{proof}
For $\lambda\in[n]$ we define $\mathcal{C}_{w,\lambda}$ to be the set of $x\in\{0,1\}^{n}$ with weight $w$ and $\lVert{x}\rVert\equiv\lambda\pmod{n}$, where $\lVert{x}\rVert$ denotes the sum of non-zero entries' positions in $x$. For example, $\lVert(0,1,1)\rVert=2+3=5$. Because all elements in $\mathcal{C}_{w,\lambda}$ have the same weight, the pairwise Hamming distance in $\mathcal{C}_{w,\lambda}$ is even. If $x,y\in\mathcal{C}_{w,\lambda}$ and $\mathrm{dist}(x,y)=2$, then $\lVert{x}\rVert\not\equiv\lVert{y}\rVert\pmod{n}$, a contradiction. Namely, the pairwise Hamming distance in $\mathcal{C}_{w,\lambda}$ is at least $4$. Furthermore, $\lvert\bigcup_{\lambda\in[n]}\mathcal{C}_{w,\lambda}\rvert=\binom{n}{w}$, since each element from $\{0,1\}^{n}$ of weight $w$ is in $\mathcal{C}_{w,\lambda}$, for some $\lambda$. Let $\mathcal{C}_{w}$ be the largest among all $C_{w,\lambda}$. By averaging we have $\lvert\mathcal{C}_{w}\rvert\geq\binom{n}{w}/n$.
\end{proof}

\vspace{0.03em}\subsection{Proof of Theorem \ref{thm4}}
\begin{proof}[Proof of Theorem \ref{thm4}]
We start by showing the stated upper bound on $\mu_{t}(Q_{n})$. Let $M\subseteq\{0,1\}^{n}$ be a maximum total mutual-visibility set in $Q_{n}$. By Lemma \ref{lem5} we know that $M$ avoids the Hamming distance $2$. We define a bipartite graph $G$ on the vertex set $M\cup\left(\{0,1\}^{n}\backslash{M}\right)$, such that $x\in{M}$ and $y\in\{0,1\}^{n}\backslash{M}$ are adjacent if and only if $\mathrm{dist}(x,y)=2$. We shall count the number of edges in $G$. First, for every $x\in{M}$, there are $\binom{n}{2}$ vertices in $\{0,1\}^{n}$ with Hamming distance $2$ from $x$, that is, the number of ways to flip two coordinates of $x$. Since $M\subseteq\{0,1\}^{n}$ avoids the Hamming distance $2$, all such vertices are in $\{0,1\}^{n}\backslash{M}$, namely,
\begin{equation}
\label{equ3}
e(G)=\sum_{x\in{M}}d_{G}(x)=\lvert{M}\rvert\binom{n}{2}.
\end{equation}
On the other hand, we claim that $d_{G}(y)\leq\lfloor\frac{n}{2}\rfloor$ holds for every $y\in\{0,1\}^{n}\backslash{M}$. Suppose there exists some $y\in\{0,1\}^{n}\backslash{M}$ with $d_{G}(y)>\lfloor\frac{n}{2}\rfloor$. Each neighbor of $y$ is obtained by flipping two coordinates of $y$. Thus, by the pigeonhole principle, there must exist two neighbors $u,v\in{M}$ of $y$ and distinct $i,j,k\in[n]$, such that $u$ is obtained by flipping the $i$-th and $k$-th coordinates of $y$, and $v$ is obtained by flipping the $j$-th and $k$-th coordinates of $y$. But then the Hamming distance between $u$ and $v$ is exactly $2$, a contradiction. Therefore,
\begin{equation*}
e(G)=\sum_{y\in\{0,1\}^{n}\backslash{M}}d_{G}(y)\leq\frac{\left(2^{n}-\lvert{M}\rvert\right)n}{2}.
\end{equation*}
Together with \eqref{equ3}, we deduce the upper bound $\mu_{t}(Q_{n})=\lvert{M}\rvert\leq\frac{2^{n}}{n}$.

\vspace{1em}
Now we prove the lower bounds on $\mu_{t}(Q_{n})$.\\
\textbf{Case 1:} $n=2^{m}-1$ with $m\in\mathbb{N}$.\\
If $m=1$, then $Q_{n}$ is isomorphic to $K_{2}$. In this case, the whole vertex set of $Q_{n}$ is a total mutual-visibility set, hence, $\mu_{t}(Q_{n})\geq2^{n}/(n+1)$. If $m\geq2$, then Lemma \ref{lem6} implies that there exists a subset $M\subseteq\{0,1\}^{n}$ with pairwise Hamming distance at least $3$ and $\lvert{M}\rvert=2^{n}/(n+1)$. By Lemma \ref{lem5}, such an $M$ is a total mutual-visibility set in $Q_{n}$, thus, $\mu_{t}(Q_{n})\geq2^{n}/(n+1)$.\\
\\\textbf{Case 2:} $n$ is an arbitrary natural number.\\
For each $w\in[0,n]$, consider a set $\mathcal{C}_{w}$ of elements from $\{0,1\}^{n}$ with weight $w$ and pairwise Hamming distance at least $4$, such that $\lvert\mathcal{C}_{w}\rvert\geq\binom{n}{w}/n$. Such a set exists by Lemma \ref{lem7}. Furthermore, we define
\begin{equation*}
\mathcal{A}:=\bigcup_{\substack{w\in[0,n]\\w\in\{0,1\}\pmod{4}}}\mathcal{C}_{w}\quad\text{and}\quad\mathcal{B}:=\bigcup_{\substack{w\in[0,n]\\w\in\{2,3\}\pmod{4}}}\mathcal{C}_{w}.
\end{equation*}
It is easy to see that $\mathcal{A}$ avoids the Hamming distance $2$. Indeed, for any two elements $x,y\in\{0,1\}^{n}$ with weight $w(x)\leq{w(y)}$, if $\mathrm{dist}(x,y)=2$, then either $w(x)=w(y)$ or $w(x)-w(y)=2$. If $w(x)=w(y)$, then $x$ and $y$ must belong to the same $\mathcal{C}_{w}$, meaning that $\mathrm{dist}(x,y)\geq4$. If $w(x)-w(y)=2$, then $w(x)\in\{0,1\}\pmod{4}$ and $w(y)\in\{0,1\}\pmod{4}$ can not hold simultaneously, a contradiction.\\
By symmetry, $\mathcal{B}$ also avoids the Hamming distance $2$. Then, Lemma \ref{lem5} ensures that both $\mathcal{A}$ and $\mathcal{B}$ are total mutual-visibility sets in $Q_{n}$. Furthermore, since $\lvert\mathcal{A}\cup\mathcal{B}\rvert=\sum_{w=0}^{n}\lvert\mathcal{C}_{w}\rvert\geq2^{n}/n$, we have $\mu_{t}(Q_{n})\geq\max\left\{\lvert\mathcal{A}\rvert,\lvert\mathcal{B}\rvert\right\}\geq2^{n-1}/n$.
\end{proof}

\section{Concluding remarks}
\label{conclusions}
In this paper we establish tight bounds (up to a multiplicative constant) on $\mu(Q_{n})$ and $\mu_{t}(Q_{n})$. Moreover, our proofs provide explicit constructions of nearly optimal mutual-visibility and total mutual-visibility sets in hypercubes.

A natural generalisation of the hypercube $Q_n$ is the so-called \textit{Hamming graph} $H(n,q)$ with $n,q\in\mathbb{N}$, which is defined as a graph on the vertex set $[q]^{n}$, where two vertices $x,y\in[q]^{n}$ form an edge if and only if their Hamming distance equals $1$. Bujt\'{a}s et al.\cite{bujta2023total} proved non-trivial bounds for $\mu_{t}\left(H(n,q)\right)$. It might be possible to improve their lower bound using certain probabilistic techniques and non-binary error correcting codes, see, e.g., Lindstr\"{o}m\cite{lindstrom1969group} and Sch\"{o}nheim\cite{schonheim1968linear}. On the other hand, the parameter $\mu\left(H(n,q)\right)$ for $n>2$ has not been studied to the best of our knowledge and might be interesting to investigate.

For the coloring version of the mutual-visibility problem, we have proved that $\omega(1)=\chi_{\mu}(Q_{n})=O(\log\log{n})$. It would be interesting to determine the order of magnitude of $\chi_{\mu}(Q_{n})$. In addition, $Q_{n}$ is the only known class of graphs with $\chi_{\mu}(Q_{n})\gg\lvert{V(Q_{n})}\rvert/\mu(Q_{n})$. Are there other classes of graphs that exhibit the same behavior?

Lastly, we note that our proofs in Section \ref{totalmutual} can be adapted to settle the coloring version of the total mutual-visibility problem for hypercubes. Let $\chi_{\mu}^{\mathrm{total}}(G)$ be the smallest number of colors needed to color $V(G)$ such that each color class is a total mututal-visibility set in $G$. In the proof of Lemma \ref{lem7} we have actually partitioned every layer of $Q_{n}$ into $n$ total mutual-visibility sets. Combining this with the lower bound proof of Theorem \ref{thm4}, one can show that $\chi_{\mu}^{\mathrm{total}}(Q_{n})\leq2n$. On the other hand, the upper bound on $\mu_{t}(Q_{n})$ implies that $\chi_{\mu}^{\mathrm{total}}(Q_{n})\geq{n}$. The function $\chi_{\mu}^{\mathrm{total}}(G)$ has not been studied elsewhere and could be of independent interest.

\vspace{3em}
\textbf{Acknowledgements.} The research was partially supported by the DFG grant FKZ AX 93/2-1. The first author thanks Gabriele Di Stefano for introducing her to the notion of mutual-visibility in graphs. The authors also thank Felix Christian Clemen and Oleg Pikhurko for helpful discussions as well as the anonymous referees for their constructive comments.

\section*{Appendix}
\label{appendix}
\newtheorem*{claim1}{\textit{Claim 1}}
\begin{claim1}
Let $n\in\mathbb{N}$ and $r,r'\in \{0,\ldots, n\}$ with $r+3\leq{r'}$. For any distinct vertices $A,B\in\mathcal{L}_{r}\cup\mathcal{L}_{r'}$, there exists a shortest $A,B$-path in $Q_{n}$ that is internally disjoint from $\mathcal{L}_{r}\cup\mathcal{L}_{r'}$ and only goes through the layers between $\mathcal{L}_{r}$ and $\mathcal{L}_{r'}$.
\end{claim1}
\begin{proof}[Proof of Claim 1]
Fix arbitrary $A,B\in\mathcal{L}_{r}\cup\mathcal{L}_{r'}$, $A\neq{B}$.\\

\noindent
\textbf{Case 1:} $\lvert{A}\rvert=\rvert{B}\rvert$.\\
Without loss of generality assume $A,B\in\mathcal{L}_{r}$, the case of $A,B\in\mathcal{L}_{r'}$ follows by a similar argument. From $\lvert{A}\rvert=\rvert{B}\rvert$ we can deduce that $\lvert{A}\backslash{B}\rvert=\lvert{B}\backslash{A}\rvert$. Let $A\backslash{B}=\{a_{1},\dots,a_{k}\}$ and $B\backslash{A}=\{b_{1},\dots,b_{k}\}$ for some $k\geq1$. If $k=1$, then $\left(A,A\cup\{b_{1}\},B\right)$ is a desired shortest $A,B$-path, see Figure \ref{fig1}.
\begin{figure}[H]
\centering
\begin{tikzpicture}[scale=7/10]
\draw[thick,blue] (-6,-1.5)--(0,-0.5)--(6,-1.5);

\draw[ultra thick,gray] (-7,-1.5)--(10,-1.5);
\node at (-7.8,-1.5){\small$\mathcal{L}_{r}$};
\draw[dashed,gray] (-7,-0.5)--(10,-0.5);
\node at (-7.8,-0.5){\small$\mathcal{L}_{r+1}$};
\draw[dashed,gray] (-7,0.5)--(10,0.5);
\node at (-7.8,0.5){\small$\mathcal{L}_{r+2}$};
\node at (-7.8,1.36){\small$\vdots$};
\draw[ultra thick,gray] (-7,2)--(10,2);
\node at (-7.8,2){\small$\mathcal{L}_{r'}$};

\fill (-6,-1.5) circle (3pt);
\node at (-6,-1.8){\small$A$};
\fill (6,-1.5) circle (3pt);
\node at (6,-1.8){\small$B$};
\fill[gray] (0,-0.5) circle (2pt);
\node at (-1,-0.17){\small$A\cup\{b_{1}\}$};
\end{tikzpicture}
\caption{A shortest $A,B$-path when $|A|=|B|$ and $k=1$}
\label{fig1}
\end{figure}

\noindent
If $k\geq2$, we first take $(A,A\cup\{b_{1}\})$. Then, we iteratively add $b_{i+1}$ and delete $a_{i}$ for $i\in[k-1]$ till we reach the vertex $A\cup\{b_{1},\dots,b_{k}\}\backslash\{a_{1},\dots,a_{k-1}\}$. Finally, we delete $a_{k}$ and reach $B$, completing the construction of a desired shortest $A,B$-path. See Figure \ref{fig2} for an illustration.

\begin{figure}[H]
\centering
\begin{tikzpicture}[scale=7/10]
\draw[thick,blue] (-6,-1.5)--(-5,-0.5)--(-4,0.5)--(-3,-0.5)--(-2,0.5)--(-1,-0.5)--(0,0.5)--(1,-0.5)--(2,0.5)--(3,-0.5)--(6,-1.5);

\draw[ultra thick,gray] (-7,-1.5)--(10,-1.5);
\node at (-7.8,-1.5){\small$\mathcal{L}_{r}$};
\draw[dashed,gray] (-7,-0.5)--(10,-0.5);
\node at (-7.8,-0.5){\small$\mathcal{L}_{r+1}$};
\draw[dashed,gray] (-7,0.5)--(10,0.5);
\node at (-7.8,0.5){\small$\mathcal{L}_{r+2}$};
\node at (-7.8,1.36){\small$\vdots$};
\draw[ultra thick,gray] (-7,2)--(10,2);
\node at (-7.8,2){\small$\mathcal{L}_{r'}$};

\fill (-6,-1.5) circle (3pt);
\node at (-6,-1.8){\small$A$};
\fill (6,-1.5) circle (3pt);
\node at (6,-1.8){\small$B$};
\fill[gray] (-5,-0.5) circle (2pt);
\node at (-6,-0.17){\small$A\cup\{b_{1}\}$};
\fill[gray] (-4,0.5) circle (2pt);
\node at (-5.23,0.83){\small$A\cup\{b_{1},b_{2}\}$};
\fill[gray] (-3,-0.5) circle (2pt);
\node at (-3,-0.87){\small$A\cup\{b_{1},b_{2}\}\backslash\{a_{1}\}$};
\fill[gray] (-2,0.5) circle (2pt);
\fill[gray] (-1,-0.5) circle (2pt);
\fill[gray] (0,0.5) circle (2pt);
\fill[gray] (1,-0.5) circle (2pt);
\fill[gray] (2,0.5) circle (2pt);
\fill[gray] (3,-0.5) circle (2pt);
\node at (6.3,-0.17){\small$A\cup\{b_{1},\dots,b_{k}\}\backslash\{a_{1},\dots,a_{k-1}\}$};
\end{tikzpicture}
\caption{A shortest $A,B$-path when $|A|=|B|$ and $k\geq2$}
\label{fig2}
\end{figure}

\noindent
\textbf{Case 2:} $\lvert{A}\rvert\neq\lvert{B}\rvert$.\\
Assume that $A\in\mathcal{L}_{r}$ and $B\in\mathcal{L}_{r'}$. It follows that $\lvert{B}\backslash{A}\rvert-\lvert{A}\backslash{B}\rvert=r'-r$. Let $B\backslash{A}=\{b_{1},\dots,b_{k}\}$ for some $k\geq{r'-r}$. If $A\backslash{B}=\emptyset$, then we obtain the desired $A$-$B$ path by simply adding elements of $B\backslash{A}$ to $A$ one at a time, see Figure \ref{fig3}.

\begin{figure}[H]
\centering
\begin{tikzpicture}[scale=7/10]
\draw[thick,blue] (-6,-1.5)--(-3,-0.5)--(0,0.5)--(6,2);

\draw[ultra thick,gray] (-7,-1.5)--(10,-1.5);
\node at (-7.8,-1.5){\small$\mathcal{L}_{r}$};
\draw[dashed,gray] (-7,-0.5)--(10,-0.5);
\node at (-7.8,-0.5){\small$\mathcal{L}_{r+1}$};
\draw[dashed,gray] (-7,0.5)--(10,0.5);
\node at (-7.8,0.5){\small$\mathcal{L}_{r+2}$};
\node at (-7.8,1.36){\small$\vdots$};
\draw[ultra thick,gray] (-7,2)--(10,2);
\node at (-7.8,2){\small$\mathcal{L}_{r'}$};

\fill (-6,-1.5) circle (3pt);
\node at (-6,-1.8){\small$A$};
\fill (6,2) circle (3pt);
\node at (6,2.3){\small$B$};
\fill[gray] (-3,-0.5) circle (2pt);
\node at (-4,-0.17){\small$A\cup\{b_{1}\}$};
\fill[gray] (0,0.5) circle (2pt);
\node at (-1.23,0.83){\small$A\cup\{b_{1},b_{2}\}$};
\end{tikzpicture}
\caption{A shortest $A,B$-path when  $|A|\neq |B|$ and $A\backslash{B}=\emptyset$}
\label{fig3}
\end{figure}

\noindent
If $A\backslash{B}\neq\emptyset$, let $A\backslash{B}=\{a_{1},\dots,a_{m}\}$ with $m=k-(r'-r)$. We first take $(A,A\cup\{b_{1}\})$, then we iteratively add $b_{i+1}$ and delete $a_{i}$ for $i\in[m]$ till we reach the vertex $A\cup\{b_{1},\dots,b_{m+1}\}\backslash\{a_{1},\dots,a_{m}\}$, afterwards we add $b_{m+2},\dots,b_{k}$ one at a time and eventually reach $B$. This constructs the shortest $A,B$-path, see Figure \ref{fig4} for an illustration.
\end{proof}

\begin{figure}[H]
\centering
\begin{tikzpicture}[scale=7/10]
\draw[thick,blue] (-6,-1.5)--(-5,-0.5)--(-4,0.5)--(-3,-0.5)--(-2,0.5)--(-1,-0.5)--(0,0.5)--(1,-0.5)--(2,0.5)--(3,-0.5)--(6,2);

\draw[ultra thick,gray] (-7,-1.5)--(10,-1.5);
\node at (-7.8,-1.5){\small$\mathcal{L}_{r}$};
\draw[dashed,gray] (-7,-0.5)--(10,-0.5);
\node at (-7.8,-0.5){\small$\mathcal{L}_{r+1}$};
\draw[dashed,gray] (-7,0.5)--(10,0.5);
\node at (-7.8,0.5){\small$\mathcal{L}_{r+2}$};
\node at (-7.8,1.36){\small$\vdots$};
\draw[ultra thick,gray] (-7,2)--(10,2);
\node at (-7.8,2){\small$\mathcal{L}_{r'}$};

\fill (-6,-1.5) circle (3pt);
\node at (-6,-1.8){\small$A$};
\fill (6,2) circle (3pt);
\node at (6,2.3){\small$B$};
\fill[gray] (-5,-0.5) circle (2pt);
\node at (-6,-0.17){\small$A\cup\{b_{1}\}$};
\fill[gray] (-4,0.5) circle (2pt);
\node at (-5.23,0.83){\small$A\cup\{b_{1},b_{2}\}$};
\fill[gray] (-3,-0.5) circle (2pt);
\node at (-3,-0.87){\small$A\cup\{b_{1},b_{2}\}\backslash\{a_{1}\}$};
\fill[gray] (-2,0.5) circle (2pt);
\fill[gray] (-1,-0.5) circle (2pt);
\fill[gray] (0,0.5) circle (2pt);
\fill[gray] (1,-0.5) circle (2pt);
\fill[gray] (2,0.5) circle (2pt);
\fill[gray] (3,-0.5) circle (2pt);
\node at (5,-0.87){\small$A\cup\{b_{1},\dots,b_{m+1}\}\backslash\{a_{1},\dots,a_{m}\}$};
\end{tikzpicture}
\caption{A shortest $A,B$-path when $|A|\neq |B|$ and  $A\backslash{B}\neq\emptyset$}
\label{fig4}
\end{figure}

\begin{proof}[Proof of Lemma \ref{lem3}]
Let $\mathcal{F}\subseteq\binom{[n+1]}{r}$ be a largest $\mathcal{D}_{r}(s,t)$-free family, namely, $\lvert\mathcal{F}\rvert=\mathrm{ex}\left(n+1,\mathcal{D}_{r}(s,t)\right)$. For every $A\in\binom{[n+1]}{n}$, we define $\mathcal{F}_{A}:=\mathcal{F}\cap\binom{A}{r}$. Since $\mathcal{F}_{A}$ is a $\mathcal{D}_{r}(s,t)$-free family with ground set of size $n$, we have $\lvert\mathcal{F}_{A}\rvert\leq\mathrm{ex}\left(n,\mathcal{D}_{r}(s,t)\right)$. On the other hand, since every member of $\mathcal{F}$ is contained in $\binom{n+1-r}{n-r}=n+1-r$ different $\mathcal{F}_{A}$'s,  we can bound $\mathrm{ex}\left(n+1,\mathcal{D}_{r}(s,t)\right)$ as follows
\begin{equation*}
\mathrm{ex}\left(n+1,\mathcal{D}_{r}(s,t)\right)=\lvert\mathcal{F}\rvert=\frac{\sum_{A\in\binom{[n+1]}{n}}\lvert\mathcal{F}_{A}\rvert}{n+1-r}\leq\frac{(n+1)\mathrm{ex}\left(n,\mathcal{D}_{r}(s,t)\right)}{n+1-r}.
\end{equation*}
Consequently, we have
\begin{equation*}
\frac{\mathrm{ex}\left(n,\mathcal{D}_{r}(s,t)\right)}{\binom{n}{r}}\geq\frac{(n+1-r)\mathrm{ex}\left(n+1,\mathcal{D}_{r}(s,t)\right)}{(n+1)\binom{n}{r}}=\frac{\mathrm{ex}\left(n+1,\mathcal{D}_{r}(s,t)\right)}{\binom{n+1}{r}}.
\end{equation*}

\vspace{1em}
To show the second monotonicity, let $\mathcal{F}\subseteq\binom{[n+1]}{r+1}$ be a largest $\mathcal{D}_{r+1}(s,t)$-free family. For every $x\in[n+1]$, we define $\mathcal{F}_{x}:=\left\{A\backslash\{x\}:\,A\in\mathcal{F}\text{ with }x\in{A}\right\}$. It is a simple observation that $\mathcal{F}_{x}$ is a $\mathcal{D}_{r}(s,t)$-free family with ground set of size $n$, so we have $\lvert\mathcal{F}_{x}\rvert\leq\mathrm{ex}\left(n,\mathcal{D}_{r}(s,t)\right)$. Moreover, since every member of $\mathcal{F}$ contributes one member to $r+1$ different $\mathcal{F}_{x}$'s, we obtain the following bound
\begin{equation*}
\mathrm{ex}\left(n+1,\mathcal{D}_{r+1}(s,t)\right)=\lvert\mathcal{F}\rvert=\frac{\sum_{x\in[n+1]}\lvert\mathcal{F}_{x}\rvert}{r+1}\leq\frac{(n+1)\mathrm{ex}\left(n,\mathcal{D}_{r}(s,t)\right)}{r+1}.
\end{equation*}
Accordingly,
\begin{equation*}
\frac{\mathrm{ex}\left(n,\mathcal{D}_{r}(s,t)\right)}{\binom{n}{r}}\geq\frac{(r+1)\mathrm{ex}\left(n+1,\mathcal{D}_{r+1}(s,t)\right)}{(n+1)\binom{n}{r}}=\frac{\mathrm{ex}\left(n+1,\mathcal{D}_{r+1}(s,t)\right)}{\binom{n+1}{r+1}}.
\end{equation*}
Taking the limit for $n\to\infty$, it holds that $\pi\left(D_{r}(s,t)\right)\geq\pi\left(\mathcal{D}_{r+1}(s,t)\right)$.
\end{proof}
\end{document}